\documentclass[runningheads]{llncs}
\usepackage[T1]{fontenc}
\usepackage{graphicx}

\usepackage{inputenc}
\usepackage{amssymb}
\usepackage{verbatim}
\usepackage{blkarray}
\usepackage{multirow}
\usepackage{graphicx}
\usepackage{afterpage}
\usepackage[hidelinks]{hyperref}
\usepackage{framed}
\usepackage{amsmath}
\usepackage{systeme}
\usepackage{xcolor}
\usepackage[margin=1.3in]{geometry}

\newcommand{\ket}[1]{|#1\rangle}
\newcommand{\bra}[1]{\langle #1|}
\newcommand{\braket}[2]{\langle#1|#2\rangle}
\newcommand{\vp}{\varphi}

\newcommand{\rk}{\operatorname{rk}}
\newcommand{\supp}{\operatorname{supp}}
\newcommand{\wt}{\operatorname{wt}}

\newcommand{\bs}[1]{\{0,1\}^{#1}}

\newcommand{\bff}[2]{f:\bs{n}\to\bs{m}}
\newcommand{\f}[2]{\mathbb{F}_{#1}^{#2}}
\newcommand{\w}[1]{\mathbf{#1}}

\DeclareMathOperator{\GPK}{GPK}
\begin{document}
\title{On the Walsh and Fourier-Hadamard Supports of Boolean Functions From a Quantum Viewpoint}
\titlerunning{Boolean Functions from a Quanutm Viewpoint}
%
\author{Claude Carlet\inst{1,2} \and
Ulises Pastor-Díaz\inst{3} \and
José M. Tornero\inst{3}}
\authorrunning{C. Carlet et al.}
%
\institute{University of Bergen, Department of Informatics, 5005 Bergen, Norway \and University of Paris 8, Department of Mathematics, 93526 Saint-Denis, France \\\email{claude.carlet@gmail.com} \and
University of Sevilla, Departament of Algebra, 41012 Sevilla, Spain \\ \email{upastor@us.es}, \email{tornero@us.es}}
\maketitle              
\begin{abstract}
In this paper, we focus on the links between Boolean function theory and quantum computing. In particular, we study the notion of what we call fully-balanced functions and analyse the Fourier--Hadamard and Walsh supports of those functions having such property. We study the Walsh and Fourier supports of other relevant classes of functions, using what we call balancing sets. This leads us to revisit and complete certain classic results and to propose new ones.

We complete our study by extending the previous results to pseudo-Boolean functions (in relation to vectorial functions) and giving an insight on its applications in the analysis of the possibilities that a certain family of quantum algorithms can offer.

\end{abstract}
\begin{section}{Introduction}

The main results of this paper deal with Boolean functions and do not require a knowledge in quantum computing and quantum algorithms, but they have been highly motivated by and have important applications in the analysis of the Generalised Phase Kick-Back quantum algorithm (a quantum algorithm inspired by the phase kick-back technique that the second and third authors introduced in \cite{gpk} and which is used to distinguish certain classes of functions). For this reason, we will begin this introduction by giving some notions about this model of computation and its relation to Boolean functions, specially for those readers coming from a Boolean function background. However, for a more general and in-depth explanation, \cite{oyt,kaye} can be consulted. 

A quantum computer is made up of qubits, which are Hilbert spaces of dimension two. A state of the qubit is a vector $\ket{\psi} = \alpha \ket{0} + \beta \ket{1},$ where $\alpha,\beta\in\mathbb{C}$, and which satisfies the normalisation condition $|\alpha|^2+|\beta|^2 = 1.$
Here, $\ket{0} = \begin{pmatrix} 1 & 0\end{pmatrix}^t$ and $\ket{1} = \begin{pmatrix} 0 & 1\end{pmatrix}^t$ are the column vectors of the canonical basis, and $\alpha$ and $\beta$ are called the amplitudes of the state. We are using the Dirac or so-called \emph{bra-ket} notation. In this notation, vectors are represented by kets, $\ket{\cdot}$, their duals (in the linear algebra sense) are denoted by bras, $\ket{\cdot}^* = \bra{\cdot}$, and the inner product is denoted by a bracket, $\braket{\cdot}{\cdot}$. Systems of multiple qubits are constructed using the tensor product (more particularly, the Kronecker specialisation) of the individual qubit systems, and a state is said to be entangled if it does not correspond to a pure tensor in said product (that is, it cannot be written as the Kronecker product of vectors in the individual qubit systems). In particular, elements of the canonical basis (called computational basis in this context) are represented using elements of $\f{2}{n}$ inside the ket: $\ket{\w{x}}_n = \bigotimes_{i=1}^n \ket{x_i},$ where $\w{x} = (x_1, x_2,\ldots, x_n) \in \f{2}{n}$. A general state of a system of $n$ qubits can be then written as $\ket{\psi}_n = \sum_{\w{x}\in\f{2}{n}} \alpha_{\w{x}}\ket{\w{x}}, \text{ where } \alpha_{\w{x}}\in\mathbb{C} \text{ for all } \w{x}\in\f{2}{n},
$ satisfying the normalisation condition $\sum_{\w{x}\in\f{2}{n}} |\alpha_{\w{x}}|^2 = 1.$ These qubit systems evolve by means of unitary matrices, and a quantum algorithm consists of the application of a unitary transformation to the first element of the computational basis (that is, the $\ket{\w{0}}_n$ vector) in a system of $n$ qubits and measuring the resulting state. We should recall that $\ket{\w{0}}_n = \otimes_{i=1}^n \ket{0}$ is not the zero vector. Indeed, vectors in the computational basis of an $n$-qubit system will be labeled using the elements of $\f{2}{n}$, and $\ket{\w{0}}_n$ will be just the first of them. The process of measuring a state, $\ket{\psi}_n = \sum_{\w{x}\in\f{2}{n}} \alpha_{\w{x}}\ket{\w{x}}$, makes it collapse into one of the elements of the computational basis, say $\ket{\w{y}}$, with probability $|\alpha_{\w{y}}|^2$, which in turn would give us $\w{y}\in\f{2}{n}$ as a result.

The Walsh transform has a deep relation to some quantum algorithms, like the Deutsch--Jozsa algorithm \cite{dyj} or the Bernstein--Vazirani algorithm \cite{byv2}. In fact, in the final superposition of these algorithms before measuring, the amplitude of a given state of the computational basis, $\ket{\w{z}}_n$ is
$$
\alpha_{\w{z}} = \frac{1}{\sqrt{2^n}}\sum_{\w{x}} (-1)^{f(\w{x})\oplus\w{x}\cdot\w{z}},
$$
where $f:\f{2}{n}\to\f{2}{}$ is the function used as an input in the algorithms. This, in particular, allows us to distinguish balanced functions from constant ones, as the Walsh transform of a balanced function---which coincides with the unnormalised amplitudes of the final superposition---takes value zero when evaluated at the zero point ($\alpha_{\w{0}} = 0$), while a constant function takes value $2^n$ ($\alpha_{\w{0}} = 1$). This implies that, in the constant situation, we will always obtain zero as the result of our measurement, while in the balanced situation we will always obtain a value different from zero. The relevance of studying these specific classes of functions springs from the fact that they are completely distinguishable using this technique, but is also due to its implications in quantum complexity theory \cite{bb1}. However, different classes of functions can be considered.

The technique used in these algorithms, called the phase kick-back, and its generalised version, the Generalised Phase Kick-Back or $\GPK$ \cite{gpk}, make this relation even more relevant. Indeed, if we use a vectorial function $F:\f{2}{n}\to\f{2}{m}$ as an input, then, after choosing $\w{y}\in\f{2}{m}$, which in this context is called a marker, the amplitudes of the states of the canonical basis in the final superposition of the $\GPK$ algorithm are:
$$
\alpha_{\w{z}} = \frac{1}{\sqrt{2^n}}\sum_{\w{x}} (-1)^{F(\w{x})\cdot\w{y}\oplus\w{x}\cdot\w{z}}.
$$
This is, once again, a normalised version of the Walsh transform (in this case of a vectorial function) and thus it seems clear that we can use the properties of the Walsh transform in distinguishing classes of functions using the $\GPK$. What is more, we will see in Section \ref{quant} that, for a particular class of functions, there is a relation between the Walsh transform of a vectorial function, $F$, and the Fourier--Hadamard support of the pseudo-Boolean function determined by the image of $F$. The aforementioned class of functions, which will be defined in Section \ref{FBsect}, will be referred to as the class of fully balanced functions, and the relation that we have pointed out can be used to solve the problem of determining the image of a fully balanced function when it is given as a black box using the $\GPK$. However. we do so in a different article \cite{gpk2}. We will proceed now to study the Fourier--Hadamard and Walsh transforms of vectorial functions.

\end{section}

\begin{section}{Preliminaries}

\begin{subsection}{Notation}

Throughout the whole paper we will work with vectors in $\f{2}{n}$, which we will write in bold. In particular, $\w{0}$ will denote the zero vector. A subset of $\f{2}{n}$ will be called a Boolean set. Regarding the different operations, we will use $\oplus$ when dealing with additions modulo $2$, but for additions either in $\mathbb{Z}$ or in $\f{2}{n}$ we will make use of $+$. Furthermore, we denote by $\cdot$ the usual inner product in $\f{2}{n}$, $\w{x}\cdot\w{y} = \bigoplus_{i=1}^{n} x_iy_i, \mbox{ for } \w{x},\w{y}\in\f{2}{n}.$

We will refer to mappings $f:\f{2}{n}\to \f{2}{}$ as Boolean functions, mappings $f:\f{2}{n}\to\mathbb{R}$ as pseudo-Boolean functions (in particular, a Boolean function can be seen as a pseudo-Boolean function) and mappings $F:\f{2}{n}\to\f{2}{m}$ as $(n,m)$-functions. Some important concepts regarding a Boolean function $f:\f{2}{n}\to \f{2}{}$ are the following. Its support, $\supp(f) = \{\w{x}\in\f{2}{n}\mid f(\w{x}) = 1\}$. Its Hamming weight, denoted by $\wt(f)$, will be the number of vectors $\w{x}\in\f{2}{n}$ such that $f(\w{x})= 1$. In other words, $\wt(f) = |\supp(f)|$. Its sign function is the integer-valued function $\chi_f(\w{x}) = (-1)^{f(\w{x})} = 1-2f(\w{x})$. Note also that $f$, can always be expressed uniquely as follows:
$$
f(\w{x}) = \bigoplus_{\w{u}\in\f{2}{n}} a_{\w{u}} \w{x}^{\w{u}}, \mbox{ where } \w{x}^{\w{u}} = \displaystyle\prod_{i=1}^{n}x_i^{u_i}.
$$ 
This expression is called the algebraic normal form, or ANF$(f)$. The degree of this polynomial is called the algebraic degree of $f$. The derivative of $f$ in the direction of $\w{a}\in\f{2}{n}$ is defined as $D_{\w{a}}(f)(\w{x}) = f(\w{x})\oplus f(\w{x}+\w{a})$. Finally, we will say that a Boolean multiset is a pair $M = (\f{2}{n},m)$ where $m:\f{2}{n} \to \mathbb{Z}_{\geq 0}$ is a pseudo-Boolean function that can take the value $0$. For each $\w{x} \in \f{2}{n}$ we will call $m(\w{x})$ its multiplicity and we will denote by $S_M = \{\w{x} \in \f{2}{n} \mid m(\w{x}) > 0\}$ the support of $M$. However, we will also represent multisets by using set notation but repeating every element of a given multiset as many times as the multiplicity indicates. For a more general overview on multisets \cite{syro} can be consulted.

\end{subsection}

\begin{subsection}{Fourier--Hadamard and Walsh transforms}

We will now give a quick summary on some results for Boolean and pseudo-Boolean functions, but for a more general reference, \cite{bfc} can be consulted.

The Fourier--Hadamard transform of a pseudo-Boolean function $f:\f{2}{n}\to\mathbb{R}$ is the function:
$$
\widehat{f}(\w{u}) = \sum_{\w{x}\in\f{2}{n}} f(\w{x})(-1)^{\w{x}\cdot\w{u}}. 
$$
We will call the Fourier--Hadamard support of $f$ the set of $\w{u}\in\f{2}{n}$ such that $\widehat{f}(\w{u}) \neq 0$ and its Fourier--Hadamard spectrum the multiset of all values $\widehat{f}(\w{u})$. It is important to underline the relation between the Fourier--Hadamard transform and linear functions. If we denote $l_{\w{u}}(\w{x}) = \w{u}\cdot\w{x}$ for $\w{u}\neq \w{0}$, we have:
\begin{align*}
\widehat{f}(\w{u}) & = \sum_{\w{x}\in\f{2}{n}} f(\w{x})(1-2\w{x}\cdot\w{u}) = \wt(f) - 2\wt(f\cdot l_{\w{u}}) \\ & = \wt(f\oplus l_{\w{u}}) - \wt(l_{\w{u}}) = \wt(f\oplus l_{\w{u}}) - 2^{n-1},
\end{align*}
while $\widehat{f}(\w{0}) = \wt(f)$. Given a Boolean function $f$, we can also calculate its Walsh transform:
$$
W_f(\w{u}) = \sum_{\w{x}\in\f{2}{n}} (-1)^{f(\w{x})\oplus\w{x}\cdot\w{u}}.
$$
We will analogously call the Walsh support of $f$ the set of $\w{u}\in\f{2}{n}$ such that $W_f(\w{u})\neq 0$ and its Walsh spectrum the multiset of all values $W_f(\w{u})$. It is clear that the Walsh transform of a Boolean function $f$ is the Fourier--Hadamard transform of its sign function, which implies by the linearity of the Fourier--Hadamard transform: $W_f = 2^n\delta_{\w{0}} - 2\widehat{f},$ where $\delta_{\w{0}}$ is the indicator of $\{\w{0}\}$ and the Boolean function $f$ is viewed here as a pseudo-Boolean function. In particular, if $\w{u}\neq \w{0}$, then we have $W_f(\w{u}) = -2\widehat{f}(\w{u})$, and thus any $\w{u}\neq \w{0}$ is in the Fourier--Hadamard support if and only if it is in the Walsh support. Regarding the zero vector we need to analyse two particular situations. When $f$ is the zero function, i.e., $f(\w{x}) = \w{0}$ for all $\w{x}\in\f{2}{n}$, then $\widehat{f}(\w{0}) = 0$ but $W_f(\w{0}) = 2^n$. On the other hand, if $f$ is a balanced function, that is, $\wt(f) = 2^{n-1}$, then $\widehat{f}(\w{0}) = 2^{n-1}$ but $W_f(\w{0}) = 0$. In any other situation the Fourier--Hadamard and Walsh supports will be the same. Some important properties of the Fourier--Hadamard transform---which result in similar properties for the Walsh transform---are the inverse Fourier--Hadamard transform formula: $\widehat{\widehat{f\,}}\!\! = 2^n f,$ and Parseval's relation:
$$
\sum_{\w{u}\in\f{2}{n}}  \widehat{f}^{\,\,2}(\w{u}) = 2^{n}\sum_{\w{x}\in\f{2}{n}} f^2(\w{x}),
$$
which for Boolean functions turns into
$$
\sum_{\w{u}\in\f{2}{n}}  \widehat{f}^{\,\,2}(\w{u}) = 2^{n} |\supp(f)|,
$$
and for the Walsh transform becomes
$$
\sum_{\w{u}\in\f{2}{n}}  W_f^2(\w{u}) = 2^{2n}.
$$

The Walsh transform of a vectorial function $F:\f{2}{n}\to\f{2}{m}$ is the function $W_F:\f{2}{n}\times \f{2}{m}\to \mathbb{Z}$ defined as follows:
$$
W_F(\w{u},\w{v}) = \sum_{\w{x}\in\f{2}{n}} (-1)^{\w{v}\cdot F(\w{x})\oplus\w{u}\cdot\w{x}},
$$
where $\w{u}\in\f{2}{n}$ and $\w{v}\in\f{2}{m}.$
\end{subsection}

\begin{subsection}{Reed--Muller codes}

We will devote Section \ref{FBsect} to analysing the concept of fully balanced sets and its relation with minimum weight codewords in Reed--Muller codes. For a deeper analysis on these codes \cite{mac} or \cite{bfc} can be consulted. Given a Boolean function $f:\f{2}{m}\to \f{2}{}$, we can identify it with a vector of length $2^m$ by fixing an ordering---we will use the lexicographical ordering---in $\f{2}{m}$. Said vector, $\w{f}$, is then the vector of evaluations of $f$ for the chosen order. The Reed--Muller code of order $r$ and length $n = 2^m$, noted $\mathcal{R}(r,m)$, is the set of vectors $\w{f}$ where $f:\f{2}{m}\to\f{2}{}$ is a Boolean function of degree at most $r$. Reed--Muller codes are linear codes with minimum distance---i.e., minimum weight among its non-zero vectors--- $2^{m-r}$ and dimension $1+m+\binom{m}{2}+\ldots+\binom{m}{r}$.

\end{subsection}

\end{section}

\begin{section}{On balanced sets and Fourier--Hadamard supports}\label{basicresults}

In this section, we introduce very simple concepts, on which we will build more complex and interesting ones. The fact that a hyperplane $H_{\w{x}} = \{\w{y}\in \f{2}{n} \mid 
 \w{x}\cdot\w{y} = 0\}$ for $\w{x}$ nonzero has $2^{n-1}$ elements can be stated in a way reminiscent of the Fourier--Hadamard transform. Given a vector space $E$ in $\f{2}{n}$, and $l_{\w{y}}$ a nontrivial linear form in $E$, then 
 $$
 \displaystyle\sum_{\w{x}\in E} (-1)^{l_{\w{y}}(\w{x})} = 0.
 $$
 We will say that $\w{y}$---the vector which determines $l_{\w{y}}$---balances $\f{2}{n}$.

\begin{definition}{($\mathbf{y}$-Balanced sets.)}\label{balset}
Let $\mathbf{y}\in\mathbb{F}_2^n$ be a nonzero binary vector, we say that a nonempty set $S\subset\mathbb{F}_2^n$ is balanced with respect to $\mathbf{y}$ or $\mathbf{y}$-balanced if:
$$
\big|( S\cap H_{\mathbf{y}} )\big| = \frac{|S|}{2}.
$$
That is, $S$ is halved by $\mathbf{y}$ with respect to the inner product.
\end{definition}

We will also say that $\w{y}$ balances $S$. If the size of $S$ is odd, then it is clear that no vector balances $S$.

\begin{remark}
There are a few equivalent ways to define this notion.

Let $\w{1}_S$ be the indicator vector of the set $S$ and $\w{l}_{\w{y}}$ the vector associated with the linear form $l_{\w{y}}$, then stating that $\w{y}$ balances $S$ is equivalent to saying that $\wt(\w{1}_S \w{l}_{\w{y}}) = |S|/2,$ as $\w{1}_S \w{l}_{\w{y}}$---the component-wise product of the two vectors, also known as the Hadamard product---is the indicator vector of $S\cap (\f{2}{n}\setminus H_{\w{y}})$.

In another equivalent way, $\w{y} \neq \w{0}$ balances $S$ if (and only if) $1_S\oplus l_{\w{y}}$ is a balanced function---i.e., a function of weight $2^{n-1}$---where $1_S$ is the indicator function of $S$ and $l_{\w{y}}$ the linear function determined by $\w{y}$, since $\wt(1_{S}\oplus l_{\w{y}}) = \wt(1_S) + \wt(l_{\w{y}})-2\wt(1_S\ l_{\w{y}})$ and $\wt(l_{\w{y}}) = 2^{n-1}$.

A third way of defining this notion, and the one we will mostly focus on, is by means of the Fourier--Hadamard transform. The nonzero vector $\w{y}\in\f{2}{n}$ balances a nonempty set $S$ if and only if:
$$
\widehat{1}_S(\w{y}) = \sum_{\w{x}\in \f{2}{n}} 1_S(\w{x})(-1)^{\w{x}\cdot\w{y}} = \sum_{\mathbf{x}\in S} (-1)^{\mathbf{x}\cdot\mathbf{y}} = 0,
$$
or equivalently $W_{1_S}(\w{y}) = 0.$
\end{remark}

In the same manner, we can define the idea of being $\mathbf{y}$-constant.

\begin{definition}{($\mathbf{y}$-Constant sets.)}\label{constset}
Let $\mathbf{y}\in\mathbb{F}_2^n\setminus \{\w{0}\}$ be a nonzero binary vector, we say that a nonempty set $S\subset\mathbb{F}_2^n$ is constant with respect to $\mathbf{y}$ or $\mathbf{y}$-constant if either
$$
S \subset H_\mathbf{y} \quad \text{ or }\quad S \cap H_\mathbf{y} = \varnothing,
$$
that is, the product $\w{x}\cdot\w{y}$ is constant for all $\w{x}\in S$. We will say that every nonempty set $S$ is $\w{0}$-constant. 
\end{definition}

\begin{remark}
We can also express this idea by means of the Fourier--Hadamard transform: given a nonempty set $S$ and its indicator function $1_S$, it is $\mathbf{y}$-constant if and only if
$$
\left|\widehat{1}_S(\w{y}) \right| = \left|\sum_{\mathbf{x}\in S} (-1)^{\mathbf{x}\cdot\mathbf{y}}\right| =  |S|.
$$
\end{remark}

Analogously, given a set $B\subset\f{2}{n}$, we will say that $S$ is $B$-balanced ($B$-constant) if it is $\w{y}$-balanced ($\w{y}$-constant) for every $\w{y}\in B$. It is important to note that the definition of $B$-constant does not require that the product $\mathbf{x} \cdot \mathbf{y}$ be the same for every pair $\mathbf{x} \in S$, $\mathbf{y} \in B$, but rather constant for $\w{x}\in S$ once a $\w{y}\in B$ is fixed.

\begin{definition}{(Balancing set and constant set.)}
Let $S\subset\mathbb{F}_2^n$ be a nonempty set, then we call its balancing set, denoted by $B(S)$, the set of all binary vectors $\mathbf{y}\in\mathbb{F}_2^n$ such that $S$ is $\mathbf{y}$-balanced and its constant set, denoted by $C(S)$, the set of all binary vectors $\mathbf{y}\in\mathbb{F}_2^n$ such that $S$ is $\mathbf{y}$-constant.
\end{definition}

In the next remark we will explain the interest of defining the balancing and constant sets in this manner.

\begin{remark}
Given a nonzero Boolean function $f:\f{2}{n}\to\f{2}{}$ with support $S = \supp(f)$, we have that its Fourier--Hadamard support is $\supp(\widehat{f}) = \f{2}{n}\setminus B(S)$. Following the relation $W_f = 2^n\delta_{\w{0}}-2\widehat{f}$, if $f$ is not a balanced function, we have that its Walsh support is $\supp(W_f) = \supp(\widehat{f}) = \f{2}{n}\setminus B(S)$. However, if $f$ is balanced, we have $\supp(W_f) = \supp(\widehat{f})\setminus \{\w{0}\} = \f{2}{n}\setminus (B(S)\cup \{\w{0}\})$.
\end{remark}

We will begin by considering the constant set problem. The following result follows from the Fourier--Hadamard transform formula.

\begin{lemma}
Let $S\subset \mathbb{F}_2^n$ and $\mathbf{s} + S = \{\mathbf{s} + \mathbf{x} \mid \mathbf{x} \in S\}$ be the translation of $S$ by $\mathbf{s}\in\mathbb{F}_2^n$, then $C(S) = C(\mathbf{s}+S)$.
\end{lemma}

Indeed, it is clear that $\widehat{1_{\w{s}+S}}(\w{y}) = (-1)^{\w{s}\cdot\w{y}} \widehat{1_S}(\w{y}).$ Using this result we can simply consider that $\mathbf{0} \in S$ without loss of generality, which simplifies things, as then, if $\mathbf{y}\in\mathbb{F}_2^n$ makes $S$ constant, it actually makes it $0$ and we can find $C(S)$ by solving a system of linear equations. In this situation, since $\w{y}$ belongs to $C(S)$ if and only if it is orthogonal to every element $\w{x}\in S$, we have:

\begin{lemma}{(Constant set.)}
Let $S$ be a set such that $\mathbf{0}\in S$, then
$$
C(S) = \bigcap_{\mathbf{x} \in S} H_\mathbf{x},
$$ 
which is the linear subspace $\langle S \rangle^{\perp}$ of dimension $n-\rk (S)$, where $\rk(S)$ (the rank of $S$) is the dimension of $\langle S \rangle$---the linear space spanned by $S$---and $\langle S \rangle^{\perp}$ is the orthogonal space of $S$ with respect to the $\cdot$ product.
\end{lemma}

Note that, still assuming that $\w{0}\in S$, we have then $C(C(S)) = \langle S \rangle$. If $\w{0}\notin S$, then it suffices to consider $S' = \w{s} + S$ for some $\w{s}\in S$. Taking now a look into the balancing set, the following properties are straightforward.

\begin{lemma}{(Properties.)}
Let $S$ be a nonempty Boolean set, $B(S)$ be as in the previous definition and let $S_1,S_2 \subset \mathbb{F}_2^n$ both $B$-balanced and nonempty, then:
\begin{itemize}
    \item[(i)] For all $A\subset B(S)$, $S$ is $A$-balanced.
    \item[(ii)] If $S_1$ and $S_2$ are such that $S_1\cap S_2 = \varnothing$, then $S_1\cup S_2$ is $B$-balanced.
    \item[(iii)] If $S_1\subset S_2$, then $S_2 \setminus S_1$ is $B$-balanced.
    \item[(iv)] $S_1\cap S_2$ is $B$-balanced if and only if $S_1\cup S_2$ is $B$-balanced.
    \item[(v)] $B(S) = B(\mathbb{F}_2^n \setminus S)$ for all $S\subset\mathbb{F}_2^n$.
    \item[(vi)] $B(\mathbb{F}_2^n) = \mathbb{F}_2^n\setminus \{\mathbf{0}\}.$
    \item[(vii)] $B(S) = B(\mathbf{s} + S)$ for all $\mathbf{s}\in\mathbb{F}_2^n$. (Invariance by translation).
\end{itemize}
\end{lemma}

We also have the following.

\begin{lemma}\label{lemaprop}
Let $S$ be a nonempty Boolean set:
\begin{itemize}
    \item[(viii)] Let $\mathbf{s} \in\f{2}{n}\setminus \{\w{0}\}$ such that $S = \mathbf{s} + S$ and $\mathbf{y} \in \mathbb{F}_2^n$ with $\mathbf{y} \cdot \mathbf{s} = 1$, then $\mathbf{y} \in B(S)$.
    \item[(ix)] Let $S$ be $B$-balanced, then $\langle S\rangle$ is $B$-balanced.
    \item[(x)] Let $S_1$ be $B_1$-balanced and $S_2$ be $B_2$-balanced, then $\langle S_1, S_2\rangle$---the vector space generated by $S_1\cup S_2$---is $(B_1 \cup B_2)$-balanced.
\end{itemize}
\end{lemma}

\begin{proof}
Property $(viii)$ follows from the fact that $\widehat{1_{\w{s}+S}}(\w{y}) = (-1)^{\w{s}\cdot\w{y}}\widehat{1_{S}}.$

For property $(ix)$, given $\w{y}\in B$, we know that there is an $\w{x}\in S$ such that $\w{y}\cdot\w{x}=1$, and by property $(viii)$ we get the result, as $\w{x} + \langle S\rangle = \langle S\rangle.$

Property $(x)$ follows from the same idea, as for any $i=1,2$; $\w{y}\in B_i$ implies that there is an $\w{x}\in S_i\subset\langle S_1\cup S_2\rangle$ such that $\w{y}\cdot\w{x}=1$.\hfill $\square$ 
\end{proof}

\begin{remark}\label{structure}
Given a nonzero Boolean function $f:\f{2}{n}\to\f{2}{}$ and let $S= \supp(f)$. Then, the constant set and the balancing set of $S$, and incidentally also the Fourier-Hadamard support of $f$, have the following structure.
\begin{enumerate}
\item Both $C(S)$ and $\widehat{f}^{-1}(|S|)$ are vector spaces, with $C(S) = \widehat{f}^{-1}(|S|)$ if $\w{0}\in S$.
\item If $\w{0}\in S$, then each of the sets $\widehat{f}^{-1}(z)$ with $z\in\mathbb{Z}$ is either empty or a union of disjoint cosets of $C(S)$, 
 this applies in particular to $B(S) = \widehat{f}^{-1}(0)$. If $r$ is the rank of $S$, then the dimension of $C(S)$ will be $n-r$ and thus we will have $2^r$ of these cosets.
 \item If $\w{0}\notin S$, taking $g(\w{x}) = f(\w{x}+\w{s})$ for some $\w{s}\in S$ we have that $\widehat{g}(\w{u}) = (-1)^{\w{s}\cdot\w{u}}\widehat{f}(\w{u}).$ This implies that now the sets $\widehat{f}^{-1}(z)\cup\widehat{f}^{-1}(-z)$ are the ones which are either empty or a union of cosets of $C(S)$, but the situation of $B(S)$ does not change.
\end{enumerate}
\end{remark}

Taking into consideration this remark, it makes sense to define the following concept.

\begin{definition}{(Balancing index.)}\label{balindex} Let $\varnothing \neq
S\subset\mathbb{F}_2^n$, then we define its balancing index to be:
$$
b(S) = \frac{|B(S)|}{|C(S)|}.
$$
This index is always an integer, and it corresponds to the amount of disjoint cosets of $C(S)$ that conform $B(S)$, as we have seen in Remark \ref{structure}.
\end{definition}

The balancing index is clearly invariant by isomorphism, but this can be taken even further

\begin{proposition}\label{isonoiso}
Let $S\subset \f{2}{n}$ be a Boolean set and  $\vp: \langle S\rangle\to \mathbb{F}_2^m$ a monomorphism (i.e., an injective linear function). Then $b(S) = b(\vp(S))$.
\end{proposition}

\begin{proof}
We have seen that both the constant and the balancing set are invariant by translation, so we will suppose that $S$ and $\vp(S)$ include the $\w{0}$ vector and that $\vp(\w{0}) = \w{0}$ without loss of generality. Let $r$ be the rank of $S$---and also of $\vp(S)$--and let $\w{s}_1,\ldots,\w{s}_r$ be independent elements of $S$, then it is clear that they conform a basis of $\langle S \rangle$, and that the vectors $\vp(\w{s}_1),\ldots,\vp(\w{s}_r)$ conform a basis of $\langle \vp(S)\rangle$. 

We also know that both the constant set and the balancing set can be computed as the sets of solutions to certain families of systems of equations of the form $\{\w{s}_i\cdot\w{x} = b_i\mid i = 1,\ldots,r\}$, where $\w{x}$ is the vector of unknowns. The vector whose $i$-th component is $b_i$ will be noted as $\w{b}\in\f{2}{r}$. If, for a certain $\w{b}\in\f{2}{r}$, the solutions to the previous system balance $S$, then the solutions to the system $\{\vp(\w{s}_i)\cdot\w{x} = b_i\mid i = 1,\ldots,r\}$ will also balance $\vp(S)$, and the same will happen in the opposite direction. Indeed, if we take any $\w{s}\in S$ such that $\w{s} = \sum_{i=1}^r \alpha_i\w{s}_i$ for some $\alpha_i \in\f{2}{}$, we have that $\vp(\w{s}) = \sum_{i=1}^r \alpha_i\vp(\w{s}_i)$. Let $\w{z}\in\f{2}{n}$ be a solution to the system $\{\w{s}_i\cdot\w{x} = b_i\mid i = 1,\ldots,r\}$ and $\w{z}'\in\f{2}{m}$ a solution to $\{\vp(\w{s}_i)\cdot\w{x} = b_i\mid i = 1,\ldots,r\}$, then
$$
\w{s}\cdot\w{z} = \left(\sum_{i=1}^r \alpha_i\w{s}_i\right)\cdot\w{z} = \bigoplus_{i=1}^r \alpha_i b_i = \left(\sum_{i=1}^r \alpha_i\vp(\w{s}_i)\right)\cdot\w{z}' = \vp(\w{s})\cdot\w{z}'.
$$

We saw in Remark \ref{structure} that the balancing sets of $S$ and $\vp(S)$ were composed of cosets of $C(S)$. Although the dimension of $C(\vp(S))$ does not have to be the same as that of $C(S)$, the amount of cosets that balance each of these sets is the same. To see this we just need to explicitly construct the bijection between the cosets of $C(S)$ and those of $C(\vp(S))$ that we have hinted at before. Each of these cosets can be assigned to the vector $\w{b}\in\f{2}{r}$ of independent terms in the system of equations whose solution is said coset. We just need to identify cosets of $C(S)$ and of $C(\vp(S))$ which are assigned to the same $\w{b}$.
\hfill $\square$ 
\end{proof}

This allows us, by taking $r = m$, to consider only the cases where $S$ is made of the zero vector, the $n$ vectors of the cannonical basis of $\f{2}{n}$ and $|S|-n-1$ other linear combinations of these vectors when studying certain properties. Indeed, let $S\subset\f{2}{n}$ be a Boolean set with rank $r\leq n$ and such that $\w{0}\in S$, and let $\w{s}_1,\ldots,\w{s}_r$ be independent elements of $S$. Then, we can consider the application $\vp:\langle S\rangle\to\f{2}{r}$ linearly determined by $\vp(\w{s}_i) = \w{e}_i$ for $i=1,\ldots,r$, where $\w{e}_i$ are the elements of the standard basis in $\f{2}{r}$. As this application satisfies the conditions presented above, we know that $b(S) = b(\vp(S))$.

\end{section}

\begin{section}{Fully balanced sets}\label{FBsect}

In this section, we will define the notion of fully balanced sets and take a look into its relation with minimal distance codewords in a Reed--Muller code. We will begin by recalling the following result that we can find, for instance, in \cite{bfc}.

\begin{proposition}\label{hadvecprop} Let $E$ be a vector subspace of $\mathbb{F}_2^n$ and $1_E$ its indicator function, then:
$$
\widehat{1}_E = |E|1_{E^{\perp}}.
$$
\end{proposition}

In particular, this implies that $C(E) = E^{\perp}$ and $B(E) = \f{2}{n} \setminus E^{\perp}$. Another interesting remark is that if $r = \dim(E)$, then $b(E) = 2^r-1$. Moreover, we will always have $B(S)\cup C(S) = \mathbb{F}_2^n$, which is the property we will use for our following definition.

\begin{definition}{(Fully balanced.)}
We say that a nonempty set $S\subset\mathbb{F}_2^n$ is fully balanced if $B(S)\cup C(S) = \mathbb{F}_2^n$.
\end{definition}

\begin{remark}
Of course, this property is equivalent to saying that $\widehat{1_S}$ is valued in $\{-|S|,0,|S|\}$, but there are actually many other ways to define it. For instance, if $r= \rk(S)$, then $S$ is fully balanced if $b(S) = 2^r-1$, as $|C(S)| = 2^{n-r}$ and $|B(S)| = 2^n -2^{n-r}$ due to Definition \ref{balindex}. However, the more intuitive one is that a nonempty $S$ is fully balanced if for every $\w{y}\in\f{2}{n}\setminus\{\w{0}\}$ we have that $|S\cap H_{\w{y}}|$ is either $|S|$, $0$ (these two cases corresponding to $y\in C(S)$ by Definition \ref{constset}) or $|S|/2$ (the case where $\w{y}$ balances $S$ using Definition \ref{balset}).
\end{remark}

To answer the question of whether there are any other fully balanced sets apart from affine spaces, we turn to \cite{mac}, and more particularly to Lemma 6 in its Chapter 13. The result we present here is just the aforementioned one, but we have rewritten it so we do not explicitly assume that the size of $S$ is a power of $2$, which is one of the premises set in \cite{mac}. This statement is not actually used in their proof, so our result is not fundamentally new, but it seems important to know that the result is more general that as stated in \cite{mac} (and we give an original and simpler proof).

\begin{theorem}\label{FBsetsTh}
Let $S\subset\f{2}{n}$ be a nonempty Boolean set and $\w{1}_S$ its indicator vector, then the following statements are equivalent.
\begin{itemize}
\item[(i)] $S$ is fully balanced.
\item[(ii)] $S$ is an affine space.
\item[(iii)] $\w{1}_S$ is a minimum weight codeword in $\mathcal{R}(r,n)$ for some $r$.
\end{itemize}
\end{theorem}

\begin{proof}
The equivalence $(ii)\iff(iii)$ can be found in \cite{mac}, and we have already taken a look into $(ii)\implies(i)$.

\noindent The implication $(i)\implies(ii)$ is also implicitly in \cite[Chapter~13]{mac}, when we read the proof of its Lemma $6$ due to Rothschild and Van Lint; indeed, the proof given in \cite{mac} does not use in fact that $|S|$ is a power of two. We will instead present another proof of this implication using the Fourier--Hadamard transform properties.

\noindent Let $1_S$ be the indicator function of $S$. As $S$ is fully balanced, we know that $\widehat{1_S}(\w{u})$ is either $0$, $|S|$ or $-|S|$. We will suppose without loss of generality that $\w{0}\in S$, as both properties (being fully balanced and being an affine space) are preserved by translation. As we have seen, in this situation the value $-|S|$ is not possible, and $C(S)$ is the vector space of those $\w{u}$ such that $\widehat{1_S}(\w{u}) = |S|$, so:
$$
\widehat{1_S} = |S|\, 1_{C(S)},
$$
where $1_{C(S)}$ is the indicator function of $C(S)$. Using now the inverse Fourier--Hadamard transform formula, we have:
$$
\widehat{\widehat{1_S}} = 2^n 1_S = |S|\, \widehat{1}_{C(S)} = |S| |C(S)|\, 1_{C(S)^{\perp}}
$$
making also use of Proposition \ref{hadvecprop}. As $1_S$ is a Boolean function, then $|S| |C(S)| = 2^n$ and $S = C(S)^{\perp}$, so $S$ is a vector space.
\hfill $\square$ 
\end{proof}

\begin{remark}
In the proof by Rothschild and Van Lint, they suppose that $|S|$ is a power of two and proceed by induction on $n$. Their proof can also be seen in terms of the Fourier--Hadamard transform, so we will briefly present it this way to show the differences between both approaches. For $n = 2$, the result is trivial, so we will suppose the result to be true for $n\leq k-1$ and prove it for $n=k$.

\noindent Let $1_S$ be the indicator function of $S$, we know that $1_S$ is fully balanced if and only if $\widehat{1_S}(\w{u})\in\{0,\pm|S|\}$ for all $\w{u}$. If there is $\w{s}\in C(S)$ which is not zero, then there is a hyperplane $H\subset\f{2}{k}$ such that $S\subseteq H$ (this $H$ is $\{\w{0},\w{s}\}^\perp$ if $\widehat{1_S}(\w{s}) = |S|$ and its complement if $\widehat{1_S}(\w{s}) = -|S|$). Taking now any hyperplane $X$ of $H$, we know that $X = H \cap H'$ for some hyperplane $H'$ in $\f{2}{n}$, and thus $S\cap X = S\cap H'$. This implies that $|S\cap X|$ verifies the induction hypothesis and we have our result. If $C(S) = \{\w{0}\}$, then Rothschild and Van Lint show that $|S| = 2^n$, and thus $S = \f{2}{n}$ by a geometrical argument counting hyperplanes, but it is simpler to do it via an equivalent argument using Parseval's relation. Using $C(S) = \{\w{0}\}$ we know that
$$
\sum_{\w{u}\in\f{2}{n}} \left(\widehat{1_S}(\w{u})\right)^2 = \left(\widehat{1_S}(\w{0})\right)^2 = |S|^2,
$$
but it is also equal to $2^n|S|$ due to Parseval's relation, so $|S|$ must be either $0$ or $2^n$.
\end{remark}

\end{section}

\begin{section}{Some Fourier--Hadamard and Walsh supports}

We now move on to analysing the Fourier--Hadamard and Walsh supports of functions that have not yet been studied from the viewpoint of the Fourier support (recall that, for any nonzero function, $B(S)$ is the complement of the Fourier--Hadamard support). A common problem that we deal with in quantum computing and that has implications in quantum complexity theory is that of distinguishing classes of functions. Indeed, if we can isolate two classes of functions (for instance, balanced and constant functions) that cannot be distinguished efficiently in one of the classical models of computation and we can distinguish them efficiently in the quantum model, the result would have interesting implications. We will show in Section \ref{quant} that the Walsh support of vectorial functions is tied to the Fourier support of the Boolean or pseudo-Boolean functions determined by their images. For that reason, knowing the Fourier support of Boolean functions is paramount in applying the $\GPK$ algorithm to distinguish classes of vectorial functions. Few results are known regarding possible Walsh supports: we know that $\f{2}{n}$ can be a Walsh support (for instance of any function having odd Hamming weight), as well as any singleton $\{\w{a}\}$ (in this latter case, this is equivalent to the fact that the function is affine). A set of the form $\f{2}{n}\setminus \{\w{a}\}$ can also be a Walsh support, a study of which can be found in \cite{car1} together with a general review in what is already known on the subject. We study now the Fourier support of a class of functions that has never been studied:

\begin{theorem}{(Balancing independent sets.)}\label{indepth}
Let $S\subset\f{2}{n}$ with $|S|$ even, $\mathbf{0}\in S$ and $r = \rk(S) = |S|-1$. Then 
$$
b(S) = \binom{r}{(r+1)/2},
$$ 
and $B(S)$ is the disjoint union of $\binom{r}{(r+1)/2}$ affine spaces of dimension $n-r$ whose underlying vector space is $C(S) = \langle S \rangle^{\perp}$.
\end{theorem}

\begin{proof}
Let $\w{s}_i\in S$; $i=1,\ldots,r$, be the nonzero elements of $S$. For a certain $\w{x}\in\f{2}{n}$ to balance $S$, it must satisfy that $\w{x}\cdot\w{s}_i = 0$ for $(r-1)/2$ of the nonzero elements of $S$ and $\w{x}\cdot\w{s}_i = 1$ for the remaining $(r+1)/2$. Let $B_{r,t}$ be the set of vectors of weight $t$ in $\f{2}{r}$, and take $t = (r+1)/2$. We know that, for every $\w{x}\in\f{2}{n}$ that balances $S$, the vector obtained after computing the $r$ possible $\w{x}\cdot\w{s}_i$ values must be in $B_{r,t}$. Next, for each element $\mathbf{b}\in B_{r,t}$, we will consider the system of equations given by $\{\mathbf{s}_i\cdot\w{x} = b_i\mid i = 1,\ldots, r\}$, where $b_i$ is the $i$-th component of $\w{b}$ and $\w{x}$ is the vector of unknowns. It is obvious that the solution to any of these systems balances $S$, as $w(\w{b}) = t$. Furthermore, any vector that balances $S$ must be a solution to one of these systems. The next step will be to analyse the solutions to these systems of equations. As the $r$ non-trivial vectors of $S$ are independent, we know that $n \geq r$, and in particular, the rank of any of the systems will be $r$ and they will always have a subspace of dimension $n-r$ as solution. The final step is simply to remark that the pairwise intersections of such subspaces are empty and that there are precisely 
$$
\displaystyle\binom{r}{(r+1)/2}
$$ 
possible such systems of equations.
\hfill $\square$ 
\end{proof}

In particular, we have that $b(S) = \displaystyle\binom{r}{(r+1)/2}.$ Let us see what this implies in terms of the Fourier--Hadamard and Walsh supports.

\begin{remark}
Let $f:\f{2}{n}\to\f{2}{}$ be a Boolean function whose support $S = \supp(f)$ satisfies the conditions of Theorem \ref{indepth}. Then, if $f$ is not balanced---which is always the case if $n>3$---the Fourier--Hadamard and Walsh supports are $\supp(\widehat{f}) = \supp(W_f) = \f{2}{n}\setminus B(S)$, and thus we have supports of size:
$$
2^n-\displaystyle\binom{r}{(r+1)/2}2^{n-r}.
$$

If $f$ is balanced then the Walsh support is of size 
$$
2^n-\displaystyle\binom{r}{(r+1)/2}2^{n-r}-1.
$$


\end{remark}

The next step in our quest could be to consider the next general situation. Let us study the case where all nontrivial elements of $S$ are independent except for one.

\begin{proposition}\label{indepmas1pr}
Let $S\subset \mathbb{F}_2^n$ with $|S|$ even, $\mathbf{0}\in S$ and $r = \rk(S) = |S| - 2$. Denoting by $\mathbf{s}_1,\ldots, \mathbf{s}_{r}$ $r$ independent elements of $S$ and by $\mathbf{s}$ the remaining nonzero element of $S$ such that
$$
\mathbf{s} = \sum_{i=1}^{r} \alpha_i \mathbf{s}_i, \text{ where } \alpha_i\in\mathbb{F}_2 \text{ for all } i\in\{1\ldots,r\},
$$
let $k = \displaystyle\sum_{i=1}^{r} \alpha_i$, where the sum is calculated in $\mathbb{Z}$, then:
$$
b(S) =  \begin{cases}
0 & \text{ if } \varphi_1(k) > \varphi_2(k) \\
\displaystyle\sum_{i=\vp_1(k)}^{\vp_2(k)} \binom{k}{i} \ \binom{r-k}{(r/2)+e(i)-i} & \text{ otherwise,}
\end{cases}
$$
where
$$
\vp(k) = (\vp_1(k),\vp_2(k)) = \begin{cases}
(0, k) & \text{ if } k < r/2\\
(1,k) & \text{ if } k = r/2\\
\displaystyle \Big(k - \frac{r}{2} + e \left( k-\frac{r}{2} \right), \frac{r}{2}+ e \left( \frac{r}{2}+1 \right)\Big) & \text{ if } k > r/2,
\end{cases}
$$
and
$$
e(x) = \frac{1+(-1)^x}{2}.
$$
Once again, $B(S)$ will be a disjoint union of $b(S)$ affine spaces with $C(S)$ as their underlying vector space.
\end{proposition}

\begin{proof}
As announced after Proposition \ref{isonoiso}, we will consider $n = r$, $\mathbf{e}_j\in S$ for each $j=1,\ldots,r$ and $\mathbf{s} = 1^k \ 0^{r-k}$ without loss of generality. In this situation, $k$ is the weight of $\mathbf{s}$. We know that $b(S)$ corresponds to the amount of affine spaces with $C(S)$ as their underlying vector space that balance $S$, but since we have $C(S) = \langle S\rangle^{\perp} = \{\w{0}\}$, each of these affine spaces consists of a single vector and $|B(S)| = b(S)$. Thus, we only need to count the vectors $\w{b}\in\f{2}{r}$ that balance $S$. Let $\w{b}\in\f{2}{r}$ be any such vector, then $\w{b}\cdot\w{s}_j = \w{b}\cdot\w{e}_j = b_j$ for every $j=1,\ldots,r$ and
$$\w{b}\cdot \w{s} = \sum_{j=1}^k b_j,$$
where $b_j$ is the $j$-th component of $\w{b}$. It is clear that for $\mathbf{b}$ to balance $S$ it must either have weight $r/2$ and have an odd amount of ones among the first $k$ positions, or have weight $(r/2)+1$ and have an even amount of ones among the first $k$ positions. At this point, the only thing that remains is to count such $\mathbf{b}$ vectors. Let $i$ be the number of ones among the first $k$ positions, then there are exactly $\binom{k}{i}$ ways for them to be distributed. If $i$ is odd, then we want $\wt(\w{b}) = r/2$, so for each choice of $b_1,\ldots,b_k$ 
$$
\binom{r-k}{(r/2)-i}
$$
ways of choosing the remaining ones. If $i$ is even, then we want $\wt(\mathbf{b}) = (r/2) +1$, and then the remaining combinations will be
$$
\binom{r-k}{(r/2)+1-i}.
$$

A compact way to consider both possibilities at the same time is
$$
\binom{r-k}{(r/2)+e(i)-i},
$$
so the total number of combinations will be:
$$
\sum_i \binom{k}{i} \binom{r-k}{(r/2)+e(i)-i}.
$$

The only thing left to do is to analyse which are the possible values for $i$, which is the task of the $\vp$ function. Given $i\leq k \leq r$, the previous expression is nonzero if and only if $i\leq k$ and $r/2+e(i)-i\leq r-k$, that is, $i\geq k-r/2+e(i)$. If $k < r/2$, then it is trivial that $i$ can range between $0$ and $k$. If $k = r/2$, then it is almost the same situation, with the slight difference that $i$ cannot be $0$, as then there would have to be $(r/2) + 1$ ones in the last $r/2$ positions. However, if $k > r/2$, then the maximum range we can achieve is from $k-(r/2)$ to $r/2$. This is not always possible, as if $r/2$ is odd, then $i = (r/2)+1$ is even and it would not balance our set. The same happens with our lower bound and the parity of $k-(r/2)$. If it were even, then it would be impossible to balance our set with $i = k - (r/2)$, so we must adjust the limits of our range in the $\vp$ function accordingly.
\hfill $\square$ 
\end{proof}

As we can see, the formulae get incredibly complicated when the set lacks structure.

\begin{remark}
Let $f:\f{2}{n}\to\f{2}{}$ be a Boolean function such that its support $S = \supp(f)$ satisfies the conditions of Proposition \ref{indepmas1pr}. Then, if $f$ is not balanced---which again is always the case if $n>3$--- we have that both the Fourier--Hadamard and Walsh supports are a disjoint union of $b(S)$ affine spaces with $C(S)$ as their underlying vector space. Therefore, their total size will be $2^n - b(S) 2^{n-r}$, where $b(S)$ is as shown in Proposition \ref{indepmas1pr} and $r$ is the rank of $S$.
\end{remark}

We will focus now on functions whose support has a certain structure. There are many known ways of determining the Fourier--Hadamard and Walsh supports of Boolean functions by decomposing them. A summary of these situations can be found in \cite{bfc}, where, in particular, we can find that, given two pseudo-Boolean functions, $\psi$ and $\vp$, in $\f{2}{n}$, then, $\widehat{\vp\otimes\psi} = \widehat{\vp}\times\widehat{\psi}.$ Also, we know that $\widehat{\vp\times\psi} = \widehat{\vp}\otimes\widehat{\psi}/2^n,$ where $\otimes$ stands here for the Kronecker product. We will proceed to do a similar thing, but using a different construction based on the $0$-kernel of the function. We will first highlight the idea with an example.

\begin{remark}
Let us consider a nonempty $S\subset \mathbb{F}_2^n$ for which there is an $\mathbf{s}\in \mathbb{F}_2^n$ such that $\mathbf{s} + S = S$ and $\rk(S) = r$. Then, any $\w{y}\in \mathbb{F}_2^n \setminus H_{\mathbf{s}}$ balances $S$, but for any $\w{y}\in H_{\w{s}}$ to balance $S$ we would need it to balance $S/\langle \w{s}\rangle$. In particular, if $|S|$ is not a multiple of $4$, then $B(S) = \mathbb{F}_2^n \setminus H_{\mathbf{s}}$ and $b(S) = 2^{r-1}.$
\end{remark}

If we note as $f$ the indicator function of $S$, then what we have here is an $\w{s}$ such that $D_{\w{s}}(f) = 0$, the set of all the elements that fulfill this property is known as the $0$-kernel, $\mathcal{E}_{0}(f) = \{\w{a}\in\f{2}{n}\mid D_{\w{a}}(f) = 0\}$, and we will use it to generalise the previous result. Let us recall some of its properties first:

\begin{proposition}{(Properties.)}
Let $f:\f{2}{n}\to\f{2}{}$ be a Boolean function and $S=\supp(f)$:
\begin{itemize}
    \item[(i)] $\mathcal{E}_{0}(f)$ is a vector subspace.  
    \item[(ii)] If $\mathbf{0}\in S$, then $\mathcal{E}_0(f)\subset S$.   
    \item[(iii)] $\mathcal{E}_0(f(\w{x})) = \mathcal{E}_0\big(f(\w{x}+\w{a})\big)$ for all $\mathbf{a}\in\f{2}{n}$. 
    \item[(iv)] $\mathcal{E}_0(f) = \displaystyle\bigcap_{\mathbf{x}\in S} (\mathbf{x}+S).$    
\end{itemize}
\end{proposition}

What we are trying to do here is to construct $B(S)$ from both $E$ and the balancing set of the set of classes $\mathbf{x} + E$. So let us delve a little deeper into this idea.

\begin{definition}{(Classes modulo $\mathcal{E}_0$.)}
Let $S\subset\mathbb{F}_2^n$ be nonempty, $f$ its indicator function and $E= \mathcal{E}_0(f)$. We will consider the set of classes modulo $E$, $S / E = \{\mathbf{x} + E \mid \mathbf{x}\in S\}$ and define the balancing set of $S/E$ as:
$$
B(S/E) = B(S) \cap H, \text{ where } H = \displaystyle\bigcap_{\mathbf{s}\in E} H_{\mathbf{s}}.
$$
\end{definition}

Annoyingly, this does not help us much, since we use $B(S)$ as part of the definition, but the next result will give us an alternative way to compute $B(S/E)$.

\begin{lemma}\label{balcll}
Let $S\subset\mathbb{F}_2^n$ be nonempty, $f$ its indicator function and $E= \mathcal{E}_0(f)$, and let $S_E$ be a complete set of representatives of $S/E$. Then:
$$
B(S/E) = B(S_E)\cap H, \text{ where } H = \displaystyle\bigcap_{\mathbf{s}\in E} H_{\mathbf{s}}.
$$
\end{lemma}

\begin{proof}
Let $\mathbf{y}\in H$, stating that $\mathbf{y}$ balances $S$ is equivalent to stating that it balances the classes, as $\mathbf{y}\cdot \mathbf{s} = 0$ for every $\mathbf{s}\in F$ and thus $\mathbf{y}\cdot \mathbf{x}_1 = \mathbf{y}\cdot \mathbf{x}_2$ for $\mathbf{x}_1,\mathbf{x'}_1$ in the same class. The result follows from the fact that, in this situation, balancing a complete set of representatives is the same as balancing the classes.
\hfill $\square$
\end{proof}

Now the next result becomes obvious.

\begin{theorem}\label{cadena3th}
Let $S\subset \mathbb{F}_2^n$ be nonempty, $f$ its indicator function and $E= \mathcal{E}_0(f)$, with $ k= \dim(E)$ and $r = \rk(S)$. Then
$$
B(S) = \big( \mathbb{F}_2^n\setminus H \big) \sqcup B(S/H),
$$
where $H =\bigcap_{\mathbf{s}\in E} H_{\mathbf{s}}$ and $\sqcup$ stands for the disjoint union. Furthermore, $b(S) = 2^{r-k}(2^k-1) + b(S_E)$, where $S_E$ is a complete set of representatives of $S/E$.
\end{theorem}

\begin{proof}
As every element of $\mathbb{F}_2^n\setminus H$ balances $S$, the result just follows from:
$$
B(S) = \big[ B(S)\cap \left(\mathbb{F}_2^n\setminus H \right) \big] \sqcup \big[ B(S)\cap H \big].
$$

For the first part, $\left(B(S)\cap \left(\mathbb{F}_2^n\setminus H\right)\right) = \mathbb{F}_2^n\setminus H$, while the second is just the definition of $B(S/E)$. Regarding $b(S)$, the $2^{r-k}(2^k-1)$ part comes from $\mathbb{F}_2^n\setminus H$. For the $b(S_E)$ part, if we consider without loss of generality $\mathbf{0}\in S$, $n=r$, $\mathbf{e}_1,\ldots,\mathbf{e}_k \in E$ and $\mathbf{e}_{k+1},\ldots,\mathbf{e}_r \in S$, where the $\w{e}_i$ are the vectors of the canonical base, then each of the vectors in $b(S_E)$ comes from a compatible system of equations of rank $r$, and thus there will be $b(S_E)$ affine spaces in $B(S/H)$.
\hfill $\square$
\end{proof}

\begin{corollary}
Let $f:\f{2}{n}\to\f{2}{}$ such that $\supp (f) = S$ is as in Theorem \ref{cadena3th}, then $\supp(\widehat{f}) = H \cap \supp(\widehat{f}_{S/H})$, where $f_{S/H}$ is the indicator function of $S/H$.
\end{corollary}

\end{section}

\begin{section}{The Fourier--Hadamard Transform for multisets}

Let us devote some time to consider the Fourier--Hadamard transform for multisets, as it will become useful for analysing the Walsh transform of vectorial functions $F:\f{2}{n}\to\f{2}{m}$ through their image multisets. The main motivation for doing so is to define the concept of fully balanced function and apply the already presented results to analyse what the $\GPK$ has to offer. For a general review on multisets we refer the reader to \cite{syro}. As we will see, most of the results we have presented remain unchanged, so we will mostly just give a small sketch of the proofs. However, we do present some new results. In particular, we answer the question of determining the possible Fourier--Hadamard supports that a pseudo-Boolean function can have using the constant and balancing sets of multisets.

\begin{definition}{(Balanced multiset.)}
Let $M = (\f{2}{n},m)$ be a nonempty Boolean multiset. We say that $\mathbf{y}\in \mathbb{F}_2^n$ balances $M$ if:
$$
\widehat{m}(\w{y}) = \sum_{\mathbf{x} \in S_M} m(\mathbf{x}) \cdot (-1)^{\mathbf{x}\cdot \mathbf{y}} = 0,
$$
and we say that it makes $M$ constant if:
$$
\left|\widehat{m}(\w{y})\right| = \left|\sum_{\mathbf{x}\in S_M} m(\mathbf{x}) \cdot (-1)^{\mathbf{x}\cdot \mathbf{y}}\right| = \displaystyle\sum_{\mathbf{x} \in S_M} m(x) = |M|,
$$
where $|M| = \displaystyle\sum_{\mathbf{x} \in S_M} m(x)$ is the cardinality of $M$. We define $C(M)$, $B(M)$ and $b(M)$ analogously.
\end{definition}

In order to solve the problem of the constant set, we only need the following result.

\begin{lemma}
Let $M$ be a Boolean multiset, then $C(M) = C(S_M).$
\end{lemma}

We will now present results analogous to those of Section \ref{basicresults} for the balancing problem.

\begin{proposition}
Let $M=(\f{2}{n},m)$ be a nonempty Boolean multiset and $\mathbf{s}\in\mathbb{F}_2^n$. Let $\mathbf{s} + M = (\f{2}{n},m')$ where $m'(\mathbf{x}) = m(\mathbf{s}+\mathbf{x})$, then $B(M) = B(\mathbf{s} + M).$
\end{proposition}

\begin{proof}
Let $\mathbf{y}\in B(M)$, then
$$
\sum_{\mathbf{x}\in S_M} m'(\mathbf{s}+\mathbf{x}) \cdot (-1)^{(\mathbf{s}+\mathbf{x})\cdot \mathbf{y}} = (-1)^{\mathbf{s}\cdot\mathbf{y}} \sum_{\mathbf{x}\in S_M} m(\mathbf{x}) \cdot (-1)^{\mathbf{x}\cdot \mathbf{y}} = 0.
$$

And thus $\mathbf{y}\in B(\mathbf{s}+M)$. As $\mathbf{s}+(\mathbf{s}+M) = M$ we have the result.
\hfill $\square$
\end{proof}

\begin{proposition}
Let $M$ be a nonempty Boolean multiset and let $C(M)$ be its constant set. Then $B(M)$ is a disjoint union of affine spaces all of them with $C(M)$ as their underlying vector space.
\end{proposition}

The reasoning is the same as the one for the set situation, and we also still have that the balancing index of a Boolean multiset is invariant by isomorphisms of $\mathbb{F}_2^n$. We will now extend the definition of fully balanced and the main results regarding this property.

\begin{definition}{(Fully balanced multiset.)}
Let $M$ be a nonempty Boolean multiset. We say that it is a fully balanced multiset if $B(M)\cup C(M) = \mathbb{F}_2^n.$
\end{definition}

\begin{theorem}
Let $M = (\f{2}{n},m)$ be a nonempty Boolean multiset; then it is fully balanced if and only if $S_M$ is an affine space and $m$ is constant.
\end{theorem}

\begin{proof}
It is clear that any such $M$ is fully balanced. The converse implication can be proven either by induction or by following the sketch in the proof of Theorem \ref{FBsetsTh} but replacing $|S|$ by $|M|$ and $1_S$ by $m$. If we did so, we would still have that if $m(\w{0})\neq 0$, then $\widehat{m} = |M|1_{C(M)},$ just as before. As the inverse Fourier--Hadamard formula still applies, we would end up with:
$$
m = \frac{|M||C(M)|}{2^n}1_{C(M)^{\perp}},
$$
having thus our result.
\hfill $\square$ 
\end{proof}

The main problem that we have been trying to solve has been that of determining the possible Fourier--Hadamard supports---or, equivalently, balancing sets---that Boolean functions can have, but what happens when we consider multisets? The final result we will reach is the following

\begin{theorem}{(Multiset balancing sizes.)}\label{multisize}
For any set $S$ such that $\w{0}\in S$ there is a multiset whose multiplicity function has $S$ as its Fourier--Hadamard support.
\end{theorem}

However, we will not prove this right away, and instead, we will first focus on a constructive method to create multisets that have increasingly smaller balancing sets.

\begin{lemma}\label{multisize2}
Let $M_1 = (\f{2}{n},m_1)$ be a nonempty multiset such that there is an $\w{y}\in B(M_1)$, then there is another multiset $M_2 = (\f{2}{n},m_2)$ satisfying $B(M_2) = B(M_1)\setminus \{\w{y}\}$.
\end{lemma}

\begin{proof}
Thanks to the inverse Fourier--Hadamard transform we know that for all $\w{x}\in\f{2}{n}$
$$
m_1(\w{x}) = \frac{1}{2^n}\widehat{\widehat{m_1}}(\w{x}).
$$

If we want to change some values in $\widehat{m_1}$, the only two conditions that $\widehat{\widehat{m_1}}(\w{x})$ must satisfy in order to determine a multiset are the following: first, it is clear that $\widehat{\widehat{m_1}}(\w{x})$ must be nonnegative for all $\w{x}\in\f{2}{n}$, and secondly, it must be a multiple of $2^n$. If we set then
$$
\widehat{m_2}(\w{x}) = \begin{cases} \widehat{m_1} +2^{n-1} & \text{if } \w{x} = \w{0} \text{ or } 
 \w{x} = \w{y} \\ \widehat{m_1}(\w{x}) & \text{otherwise,}  \end{cases}
$$
it is clear that 
$$
\widehat{\widehat{m_2}}(\w{x}) = \begin{cases} \widehat{\widehat{m_1}}(\w{x}) & \text{if } \w{x}\cdot\w{y} = 1 \\ \widehat{\widehat{m_1}}(\w{x}) + 2^n & \text{if } \w{x}\cdot\w{y} = 0 \end{cases}
$$
and thus both conditions are satisfied. As $M_1$ was nonempty, $\widehat{m_1}(\w{0})$ was not $0$, so $B(M_2) = B(M_1)\setminus\{\w{y}\}$.
\hfill $\square$ 
\end{proof}

We will see this with an example as soon as we finish the proof of the general result.

\begin{proof}{(Theorem \ref{multisize}.)}
We know multisets in $\f{2}{n}$ whose Fourier--Hadamard support is just $\{\w{0}\}$---for instance $\f{2}{n}$---so we only need to iterate the procedure of Lemma \ref{multisize2} to obtain the desired multiset, but we can actually just begin with the empty multiset.
\hfill $\square$ 
\end{proof}


\end{section}

\begin{section}{Application in Quantum Computing}\label{quant}

We will finally point out why it seems relevant to study the Fourier--Hadamard and Walsh transforms for multisets and in the fully balanced situation. This will be done by tying down the Walsh transform of vectorial functions and some results in quantum computing. As we already mentioned in the introduction, the Walsh transform appears in quantum algorithms in which the phase kick-back technique \cite{kaye} is applied. In particular, it first appeared in the Deutsch--Jozsa algorithm \cite{dyj}, where the final amplitude associated to a certain state $\ket{\w{x}}_n$ when the algorithm is applied to a Boolean function $f:\f{2}{n}\to\f{2}{}$ is shown to be $W_f(\w{x})/2^n$.

The Walsh transform of vectorial functions also shows up in the generalised version of this algorithm \cite{gpk}. Indeed, once again, the final amplitude associated to a certain state $\ket{\w{x}}_n$ when the algorithm is applied to a vectorial function $F:\f{2}{n}\to\f{2}{m}$ using a marker $\w{y}\in\f{2}{m}$ is shown to be $W_F(\w{x},\w{y})/2^n$. We will show now the relation between fully balanced multisets and the Walsh transform of vectorial functions. To do so, we will first define some concepts.

\begin{definition}{($\mathbf{y}$-Balanced function.)}
Let $F:\f{2}{n}\to \f{2}{m}$ be a vectorial function and let $\mathbf{y}\in \f{2}{m}$. We say that $F$ is $\mathbf{y}$-balanced if we have $F(\mathbf{x})\cdot \mathbf{y} = 0$ for half of the vectors $\mathbf{x}\in\f{2}{n}$ and $F(\mathbf{x}) \cdot \mathbf{y} = 1$ for the other half.
\end{definition}

In the same way, we can define the idea of $\mathbf{y}$-constant functions.

\begin{definition}{($\mathbf{y}$-Constant function.)}
Let $F:\f{2}{n}\to \f{2}{m}$ be a vectorial function and let $\mathbf{y}\in \f{2}{m}$. We say that $F$ is $\mathbf{y}$-constant if for every vector $\mathbf{x}\in\f{2}{n}$ the result of $F(\mathbf{x}) \cdot \mathbf{y}$ is the same.
\end{definition}

As we have seen, these definitions can be translated in terms of the Fourier--Hadamard transform.

\begin{proposition}
Let $F:\f{2}{n}\to\f{2}{m}$ be a vectorial function. Then, $F$ is $\mathbf{y}$-balanced if and only if:
$$
\sum_{\mathbf{x}\in\f{2}{n}} (-1)^{F(\mathbf{x})\cdot\mathbf{y}} = 0.
$$

Similarly, $F$ is $\mathbf{y}$-constant if and only if:
$$
\left|\sum_{\mathbf{x}\in\f{2}{n}} (-1)^{F(\mathbf{x})\cdot\mathbf{y}} \right| = 2^n.
$$
\end{proposition}

The next result now becomes trivial.

\begin{theorem}
Let $F:\f{2}{n}\to\f{2}{m}$ be a vectorial function and $\mathbf{y}\in\f{2}{m}$. Then, $F$ is $\mathbf{y}$-constant if and only if $\big|W_F(\w{0},\w{y})\big| = 2^n$. Besides, $F$ is $\w{y}$-balanced if and only if $W_F(\w{0},\w{y}) = 0$.
\end{theorem}

Following things up, we will now focus on the fully balanced situation.

\begin{definition}{(Fully balanced functions.)}
Let $F:\f{2}{n}\to\f{2}{m}$ be a vectorial function, we say that it is fully balanced if for every $\mathbf{y}\in\f{2}{m}$, $F$ is either $\mathbf{y}$-balanced or $\mathbf{y}$-constant.
\end{definition}

So we say that a function is fully balanced if the multiset of its image is fully balanced. We can likewise define the concepts of $C(F)$, $B(F)$ and $b(F)$ for a given vectorial function $F$. We will finish things up by highlighting a trivial consequence of this result that ties down the Walsh transform of a vectorial function and the balancing set of its image multiset. As a previous notation remark, we will refer to the image of $F$ multiset as $I_F = (\f{2}{m},m(\w{x}))$ with $m(\w{x}) = |F^{-1}(\w{x})|.$

\begin{corollary}
Let $F:\f{2}{n}\to\f{2}{m}$ be a fully balanced vectorial function. Then,
$$
\Big|W_F(\w{0},\w{v})\Big| = 2^n\, 1_{C(I_F)}(\w{v}),
$$
where $\w{v}\in\f{2}{m}$ and $1_{C(I_F)}$ is the indicator function of $C(I_F)$.
\end{corollary}

In other words, when we are dealing with fully-balanced functions, we can determine whether a marker $\w{y}$ is in the balancing set or in the constant set of its image multiset by looking at $\GPK(\w{y})$. If the result is zero, then $\w{y}$ is in fact in the constant set, while if we get any other result, then $\w{y}$ is in the balancing set. The idea now is that we can use the $\GPK$ algorithm to determine the dimension of the image of a given fully balanced function, and even compute said image, but this will be done on a separate publication.

\end{section}

\begin{credits}
\subsubsection{\ackname} The research of the first author is partly supported by the \emph{Norwegian Research Council}, ID 314395. The research of the second and third authors is supported by the \emph{Ministerio de Ciencia e Innovación} under Project PID2020-114613GB-I00 (MCIN/AEI/10.13039/501100011033).

\subsubsection{\discintname} The authors report that there are no competing interests to declare.
\end{credits}

\bibliographystyle{splncs04}
\bibliography{mybibliography}

\begin{thebibliography}{10}
\providecommand{\url}[1]{\texttt{#1}}
\providecommand{\urlprefix}{URL }
\providecommand{\doi}[1]{https://doi.org/#1}

\bibitem{byv2}
Bernstein, E., Vazirani, U.: Quantum {C}omplexity {T}heory. SIAM Journal on Computing  \textbf{26}(5),  1411--1473 (1997)

\bibitem{bb1}
Berthiaume, A., Brassard, G.: The {Q}uantum {C}hallenge to {S}tructural {C}omplexity {T}heory. In: [1992] Proceedings of the 7th Annual Structure in Complexity Theory Conference. pp. 132--137 (1992). \doi{10.1109/SCT.1992.215388}

\bibitem{bfc}
Carlet, C.: {B}oolean {F}unctions for {C}ryptography and {C}oding {T}heory. Cambridge University Press (2021)

\bibitem{car1}
Carlet, C., Mesnager, S.: On the {S}upports of the {W}alsh {T}ransforms of {B}oolean {F}unctions. Cryptology ePrint Archive, Paper 2004/256 (2004), \url{https://eprint.iacr.org/2004/256}, \url{https://eprint.iacr.org/2004/256}

\bibitem{dyj}
Deutsch, D., Jozsa, R.: Rapid {S}olution of {P}roblems by {Q}uantum {C}omputation. Proceedings of the Royal Society of London. Series A: Mathematical and Physical Sciences  \textbf{439},  553--558 (1992)

\bibitem{kaye}
Kaye, P., Laflamme, R., Mosca, M.: An {I}ntroduction to {Q}uantum {C}omputing. OUP Oxford (2007)

\bibitem{mac}
MacWilliams, F.J., Sloane, N.J.A.: The {T}heory of {E}rror {C}orrecting {C}odes. Elsevier (1977)

\bibitem{oyt}
Ossorio-Castillo, J., Tornero, J.M.: Quantum {C}omputing from a {M}athematical {P}erspective: a {D}escription of the {Q}uantum {C}ircuit {M}odel. arXiv preprint arXiv:1810.08277  (2018)

\bibitem{gpk2}
Ossorio-Castillo, J., Pastor-D{\'\i}az, U., Tornero, J.M.: Further {A}ppliactions of the {G}eneralised {P}hase {K}ick-{B}ack, {P}reprint. (2024)

\bibitem{gpk}
Ossorio-Castillo, J., Pastor-D{\'\i}az, U., Tornero, J.M.: A {G}eneralisation of the {P}hase {K}ick-{B}ack. Quantum Information Processing  \textbf{22}(3), ~143 (2023)

\bibitem{syro}
Syropoulos, A.: Mathematics of {M}ultisets. In: Calude, C.S., P{\u{a}}un, G., Rozenberg, G., Salomaa, A. (eds.) Multiset Processing. pp. 347--358. Springer Berlin Heidelberg, Berlin, Heidelberg (2001)

\end{thebibliography}

\end{document}